\documentclass[11pt, reqno]{amsart}
\usepackage{amsfonts,amssymb,latexsym,amsmath, amsxtra}
\usepackage{verbatim}

\usepackage[OT2,T1]{fontenc}
\DeclareSymbolFont{cyrletters}{OT2}{wncyr}{m}{n}
\DeclareMathSymbol{\Sha}{\mathalpha}{cyrletters}{"58}

\theoremstyle{plain}
\newtheorem{theorem}{Theorem}[section]
\newtheorem*{theorem*}{Theorem}
\newtheorem{corollary}[theorem]{Corollary}

\newtheorem{lemma}[theorem]{Lemma}
\newtheorem{proposition}[theorem]{Proposition}
\newtheorem*{conjecture*}{Conjecture}

\theoremstyle{definition}
\newtheorem{definition}[theorem]{Definition}
\theoremstyle{remark}

\newtheorem*{remark*}{Remark}
\newtheorem*{remarks*}{Remarks}

\numberwithin{equation}{section}

\newcommand{\R}{\mathbb R}
\newcommand{\N}{\mathbb N}
\newcommand{\Z}{\mathbb Z}
\newcommand{\C}{\mathbb C}

\def\({\left(}
\def\){\right)}

\newcommand{\rank}{\text{rank}}
\newcommand{\crank}{\text{crank}}
\newcommand{\spt}{\mathrm{spt}}
\newcommand{\ospt}{\mathrm{ospt}}
\newcommand{\ST}{\mathrm{ST}}

\begin{document}

\title
{Asymptotic inequalities for positive crank and rank moments}
\author{Kathrin Bringmann}
\address{Mathematical Institute\\University of
Cologne\\ Weyertal 86-90 \\ 50931 Cologne \\Germany}
\email{kbringma@math.uni-koeln.de}
\author{Karl Mahlburg}
\address{Department of Mathematics \\
Louisiana State University \\
Baton Rouge, LA \\ U.S.A.}
\email{mahlburg@math.lsu.edu}

\thanks {The research of the first author was supported by the Alfried Krupp Prize for
Young University Teachers of the Krupp Foundation.}

\date{\today}


\begin{abstract}

Andrews, Chan, and Kim recently introduced a modified definition of crank and rank moments for integer partitions that allows the study of both even and odd moments.  In this paper, we prove the asymptotic behavior of these moments in all cases, and our main result states that while the two families of moment functions are asymptotically equal, the crank moments are always asymptotically larger than the rank moments.

Andrews, Chan, and Kim primarily focused on one case, and proved the stronger result that the first crank moment is strictly larger than the first rank moment for all partitions by showing that the difference is equal to a combinatorial statistic on partitions that they named the ospt-function.  Our main results therefore also give the asymptotic behavior of the ospt-function, and we further determine its behavior modulo $2$ by relating its parity to Andrews spt-function.

\end{abstract}

\maketitle

\section{Introduction and Statement of Results}
\label{S:intro}

Ramanujan's famous congruences for the integer partition function $p(N)$ require little introduction, as they continue to inspire active research even nearly a century after his original discoveries.  He proved in \cite{Ram21} that for $N\geq 0$,
\begin{align}
\label{E:Ram}
p(5N+4)&\equiv 0\pmod 5,\\
p(7N+5)&\equiv 0\pmod 7,\notag \\
p(11N+6)&\equiv 0\pmod{11}. \notag
\end{align}
His approach relied on the modularity properties of the partition generating function
$$
\sum_{N \geq 0} p(N) q^N = \frac1{(q;q)_\infty},
$$
where for $n\in \N_0 \cup\{\infty\}$ we adopt the standard $q$-factorial notation $(a)_n=(a; q)_n := \prod_{j = 0}^{n-1} (1 - aq^j)$.  However, his use of Hecke operators on $\ell$-adic modular forms gave little combinatorial insight as to why the partition function satisfies \eqref{E:Ram}.

Combinatorial partition statistics have subsequently been used to better understand the Ramanujan congruences, beginning with Dyson's definition of \cite{Dy44} the {\it rank} of a partition $\lambda$.  This is given by
\begin{equation}\label{definerank}
\rank(\lambda) := \text{largest part of} \; \lambda - \text{number of parts
of} \; \lambda.
\end{equation}
The rank was later followed by Garvan's discovery of the {\it crank} in \cite{Gar88}.  Andrews and Garvan's subsequent reformulation for the definition of the crank in \cite{AndG88} then successfully completed the search for combinatorial decompositions of the three congruences in \eqref{E:Ram}.  If $o(\lambda)$ denotes the number of ones in
$\lambda$,
and $\mu(\lambda)$ is the number of parts strictly larger than $o(\lambda),$ then
\begin{equation}\label{definecrank}
\crank(\lambda) := \begin{cases} \text{largest part of} \; \lambda \qquad &
\text{if} \; o(\lambda) = 0, \\
\mu(\lambda) - o(\lambda) & \text{if} \; o(\lambda) > 0.
\end{cases}
\end{equation}

However, the study of these statistics does not end with the Ramanujan congruences, as there have been a great wealth of further results regarding the arithmetic properties of the crank and rank \cite{B09, Gar10, Mah05}, connections to modular and automorphic forms, including mock theta functions \cite{BGM09, BO10}, asymptotic behavior and inequalities \cite{BMR10, BMR12}, and related combinatorial objects \cite{And07, ACK12, Gar11}.  In this paper we focus particularly on the latter two topics, and use analytic and number theoretic techniques to build on ideas introduced in the recent work of Andrews, Chan, and Kim \cite{ACK12}.

Let $M(m,N)$ (resp. $N(m,N)$) be the number of partitions of $N$
with
crank (resp. rank) $m$. Then aside from the anomalous case of $M(m,N)$
when $N=1$ (where the correct combinatorial values are $M(0,1) = 1$ and $M(m,1) = 0$ for all $m \neq 0$), the
two-parameter
generating functions may be written as  \cite{AndG88, AS54}
\begin{align}
\label{E:Cgen}
C(w;q) & := \sum_{\substack{m \in \Z \\ N \geq 0}} M(m,N) w^m q^N =
\frac{1-w}{(q)_\infty} \sum_{n \in \Z} \frac{(-1)^{n} q^{n(n+1)/2}}{1-wq^n},
\\
\label{E:Rgen}
R(w;q) & := \sum_{\substack{m \in \Z \\ N \geq 0}} N(m,N) w^m q^N =
\frac{1-w}{(q)_\infty} \sum_{n \in \Z} \frac{(-1)^{n}
q^{n(3n+1)/2}}{1-wq^n}.
\end{align}
Atkin and Garvan introduced the moments of these statistics in \cite{AG03} as follows.  For $r \in \N$, the $r$-th  \textit{crank moment} is defined by
\begin{align}
\label{E:defMr}
M_r(N):=\sum_{m \in \Z} m^r\, M(m,N),
\end{align}
and the $r$-th \textit{rank moment} by
\begin{align}
\label{E:defNr}
N_r(N):=\sum_{m \in \Z} m^r\, N(m,N).
\end{align}

Atkin and Garvan described algebraic and arithmetic relationships between crank and rank moments, but the nature of their exact numerical values was first explored by Garvan, who conjectured a strong inequality between the two families of moments in \cite{Gar10}; a related set of conjectures were later given by the present authors in \cite{BM09}.  All of these conjectures were subsequently proven, and the most definitive known results can be succinctly stated as below.
\begin{theorem*}
Suppose that $r = 2k \in \N$ is even.
\begin{enumerate}
\item
As $N \rightarrow \infty$,
\begin{equation}
\label{E:MsimN}
M_{2k}(N) \sim N_{2k}(N).
\end{equation}
\item
For all $N > 0$,
\begin{equation}
\label{E:M>N}
M_{2k}(N) > N_{2k}(N).
\end{equation}
\end{enumerate}
\end{theorem*}
\begin{remark*}
In both cases, much more is known than is stated above.  The authors and Rhoades \cite{BMR10} proved that there are certain explicit constants $A_{2k}$ such that
\begin{equation}
\label{E:MsimNA}
M_{2k}(N) \sim N_{2k}(N) \sim A_{2k} N^{k} p(N).
\end{equation}
Furthermore, Garvan \cite{Gar11} showed that each of the differences $M_{2k}(N) - N_{2k}(N)$ can be written as a linear combination (with positive coefficients) of ``higher spt-functions'', which are manifestly positive counts of certain combinatorial objects.
\end{remark*}

However, the simple fact that $M(m,N)$ and $N(m,N)$ are both invariant under $m \mapsto -m$ clearly implies that $\eqref{E:defMr}$ and \eqref{E:defNr} are both identically zero when $r$ is odd.  We therefore adopt the definition proposed by Andrews, Chan, and Kim \cite{ACK12} in order to study nontrivial odd crank and rank moments.
\begin{definition}
\label{D:MN+}
Suppose that $r \in \N$.  The {\it $r$-th positive crank (resp. rank) moment} is defined by
\begin{align*}
M_r^+ (N) & := \sum_{m\in\N}m^r M(m,N) , \\
N_r^+ (N) & := \sum_{m\in\N}m^r N(m,N).
\end{align*}
\end{definition}
\begin{remark*}
Note that if $r$ is even, there is no new information gained in the study of the positive crank and rank moments, as for $k\geq 1$
\begin{align}
\label{E:MN2k}
M_{2k}^+ (N) & = \frac12 M_{2k} (N), \qquad \text{and } \quad
N_{2k}^+ (N) =  \frac12 N_{2k} (N) .
\end{align}
\end{remark*}

After introducing the definition of general positive moments, Andrews, Chan, and Kim focused on the first new case, namely, the odd moments $M_1^+(N)$ and $N_1^+(N)$.  One of their main results states that the first crank moment is always larger than the first rank moment.
\begin{theorem*}[Theorem 3 in \cite{ACK12}]
For $N > 0$,
\begin{equation}
\label{E:M1>N1}
M_1^+(N) > N_1^+(N).
\end{equation}
\end{theorem*}

In this paper, we combine the above ideas and study the relationship between all positive crank and rank moments.  Our main result precisely identifies the first two terms in the asymptotic expansion of the positive crank and rank moments.
\begin{theorem}
\label{T:maintheorem2}
Suppose that $r \in \N$.
\begin{enumerate}
\item
\label{T:maintheorem2:MsimN}
As $N \to \infty$,
\begin{align*}
M_r^+ (N) \sim N_r^+(N) & \sim  \gamma_r \: N^{\frac{r}2 -1} e^{\pi\sqrt{\frac{2N}3}}, \\
\end{align*}
where
\begin{align*}
\gamma_r &:= r! \: \zeta (r) \left( 1- 2^{1-r}\right) \frac{6^{\frac{r}2}}{4\sqrt{3} \pi^r}.
\end{align*}
Here $\zeta (s)$ denotes the Riemann $\zeta$-function.

\item
\label{T:maintheorem2:M>N}
As $N \to \infty$,
\begin{align*}
M_r^+ (N) - N_r^+ (N) &\sim  \delta_r \: N^{\frac{r}2 - \frac32} e^{\pi\sqrt{\frac{2N}3}},
\end{align*}
where
\begin{align*}
\delta_r &:= r! \: \zeta (r-2) \left( 1- 2^{3-r}\right) \frac{6^{\frac{r-1}2}}{4\sqrt{3} \pi^{r-1}}.
\end{align*}
\end{enumerate}
\end{theorem}

As a corollary, we obtain a generalization of \eqref{E:MsimN} and an asymptotic version of \eqref{E:M>N} for all positive crank and rank moments.
\begin{corollary}
\label{C:main}
Suppose that $r\in\N$.
\begin{enumerate}
\item
\label{C:main:MsimN}
As $N \to \infty$,
\begin{equation*}	
M_r^+ (N) \sim N_r^+ (N).
\end{equation*}
\item
\label{C:main:M>N}
There exists $n_r \in \N$ such that if $N \geq n_r$, then
\begin{equation*}
M_r^+(N) > N_r^+(N).
\end{equation*}
\end{enumerate}
\end{corollary}
\begin{remarks*}
{\it 1.}
For certain small $r$, the expressions for $\gamma_r$ and $\delta_r$ given in Theorem \ref{T:maintheorem2} must be evaluated using the analytic continuation of the Riemann zeta function. In particular,
\begin{gather*}
\gamma_1 = \frac{\log(2)}{2\sqrt{2} \pi}, \\
\delta_3 = \frac{3\sqrt{3} \log(2)}{\pi^2}, \qquad
\delta_2 = \frac{1}{2\sqrt{2} \: \pi}, \qquad \text{ and } \quad
\delta_1 = \frac{1}{16 \sqrt{3}}.
\end{gather*}
Although the overall result can be stated uniformly, these cases must be dealt with separately in our proofs. \\
\noindent {\it 2.}
If $r$ is even, these expressions are equivalent to those shown in \cite{BMR10}.  In their detailed study of the case $r=1$, Andrews, Chan, and Kim also proved (Theorem 6 in \cite{ACK12}) that
\begin{equation}
\label{E:M1d}
M_1^+(N) = \sum_{\lambda \vdash N} d(\lambda)
\end{equation}
where $d(\lambda)$ denotes the size of the {\it Durfee square} of a partition $\lambda$ and the summation is over all partitions of $N$.
This sum was previously studied by Canfield, Corteel, and Savage, who showed that (Corollary 3 of \cite{CCS98})
\begin{equation}
\label{E:dlambda}
\sum_{\lambda \vdash N} d(\lambda) \sim \frac{\sqrt{6} \log(2)}{\pi} \sqrt{N} p(N),
\end{equation}
which is equivalent to our asymptotic formula for $M_1^+(N).$ \\
\noindent {\it 3.} The relative error in both asymptotic expressions in Theorem \ref{T:maintheorem2} is of order $N^{-\frac{1}{2}}$.  In fact, our proof allows us to write an asymptotic expansion of the form
\begin{equation*}
M_r^+(N) = \gamma_r N^{\frac{r}{2}-1} e^{\pi\sqrt{\frac{2N}3}} \left(1 + \sum_{j = 1}^S b_j N^{-\frac{j}{2}} + O\left(N^{-\frac{S+1}{2}}\right)\right)
\end{equation*}
for any $S \geq 0$, and the same is true of the rank moments (and hence also for the differences). \\
\noindent {\it 4.}
In parallel work, Diaconis, Janson, and Rhoades \cite{DJR12} have also independently proven the main term for the positive rank moments as found in Theorem \ref{T:maintheorem2} part \ref{T:maintheorem2:MsimN}.  Their main tool was the Method of Moments, which allows the determination of the limiting distribution (in the weak sense) of the partition rank statistic as the size of the partitions grow to infinity. \\
\noindent {\it 5.} Numerical data suggests that the inequalities in Corollary \ref{C:main} part \ref{C:main:M>N} are true for all $N$, but although the constants $n_r$ can be made effective for any $r$, a uniform proof of nonasymptotic inequalities remains out of reach of our methods.  It would be of great interest to obtain a combinatorial proof, and particularly to find analogues of Garvan's higher spt-functions from \cite{Gar11} for the odd positive moments.
\end{remarks*}

It is notable that many of the previous proofs for the asymptotic behavior and arithmetic properties of even crank and rank moments relied on the modular and automorphic properties of their generating functions \eqref{E:Cgen} and \eqref{E:Rgen}.  The positive moments do not share these properties, as the singly-infinite summations in Definition \ref{D:MN+} are qualitatively lacking the symmetry of the full moments in \eqref{E:defMr} and \eqref{E:defNr}. The effect of this is seen in the presence of so-called {\it false} theta functions throughout our proofs.  We have therefore developed an entirely different approach, using Mittag-Leffler theory, the Mellin transform, and a variant of the Hardy-Ramanujan Circle Method due to Wright \cite{Wri41} in order to study the generating functions of the positive moments directly.

Wright's approach is not widely known, and gives much weaker results than Hardy and Ramanujan (whose technique was subsequently refined by Rademacher in \cite{Rad37}) in the study of the coefficients of modular forms.  Indeed, following recent further work \cite{BM11}, one can obtain formulas for the coefficients of hypergeometric series that satisfy mock modular or mock Jacobi transformation properties.  This was also previously used by the authors and Rhoades in order to obtain asymptotic series for even crank and moments with polynomial error in \cite{BMR12}.  However, these techniques do not apply to the positive moments, and the advantage of Wright's approach is that it is equally effective on any function with an asymptotic expansion.  As such, his method lies somewhere between the Hardy-Ramanujan Circle Method and Tauberian theorems (which give the asymptotic behavior of the coefficients of a power series based solely on its analytic behavior near a single singularity).  Wright's approach is powerful enough to provide an asymptotic expansion for the coefficients, but flexible enough that it applies to nonmodular generating functions such as the positive crank and rank moments.

There has also been particular interest in the smallest nontrivial crank and rank moments.
In \cite{And07}, Andrews defined the smallest parts function $\spt(N)$ as the
sum of the total number of appearances of the
smallest part in each integer
partition of $N$, and proved the surprising fact that
\begin{equation*}
\spt(N) = \frac{1}{2}\left(M_2(N) - N_2(N)\right) = M_2^+(N) - N_2^+(N).
\end{equation*}
Note that the second moments are the smallest nonzero case of the full moments.
Using Hardy and Ramanujan's famous asymptotic formula
\begin{equation}\label{pasymp}
p(N)\sim \frac{1}{4\sqrt{3}N} e^{\pi\sqrt{\frac{2N}{3}}},
\end{equation}
the $r=2$ case of Theorem \ref{T:maintheorem2}  \ref{T:maintheorem2:M>N} therefore implies that

\begin{equation*}
\spt(N) \sim \frac{\sqrt{6}}{\pi} \sqrt{N} p(N) \sim \frac{1}{2\sqrt{2} \pi \sqrt{N}} e^{\pi \sqrt{\frac{2N}{3}}},
\end{equation*}
which was first proven in \cite{B08}.

The smallest parts function has also been studied for its arithmetic properties, and satisfies three linear congruences modulo $5, 7$, and $13$ that are remarkably similar to \eqref{E:Ram} \cite{And07}.  The parity of $\spt(N)$ is equally striking, as it is entirely dictated by certain quadratic characters on the prime factorization, as was recently proven by Andrews, Garvan, and Liang \cite{AGL12}.

In the course of their study of the positive crank and rank moments, Andrews, Chan, and Kim
defined a natural counterpart to $\spt(N)$ for the first moments, setting
\begin{equation*}
\ospt(N) := M_1^+(N) - N_1^+(N).
\end{equation*}
They then proved the positivity of $\ospt(N)$ (cf. \eqref{E:M1>N1}) by showing that $\ospt(N)$ in fact counts a certain statistic on partitions.  In particular, they defined $\ST(\lambda)$ as the number of even and odd strings in a partition $\lambda$ and proved (Theorem 4 in \cite{AGL12})
\begin{equation}
\label{E:osptST}
\ospt(N) = \sum_{\lambda \vdash N} \ST(\lambda).
\end{equation}
The precise definition of $\ST$ is somewhat technical, so we do not restate it until later in this paper, where we also derive some important combinatorial properties.  Most notably, we will show that $\ospt(N)$ is a weakly increasing function in $N$.

Our next results describe the asymptotic behavior of $\ospt(N)$.
\begin{theorem}
\label{T:osptAsym}
As $N \to \infty$,
\begin{equation*}
\ospt(N) \sim \frac{1}{4} \: p(N) \sim \frac{1}{16\sqrt{3} N} e^{\pi \sqrt{\frac{2 N}{3}}}.
\end{equation*}
\end{theorem}
\begin{remark*}
The relationship between $\ospt(N)$ and $p(N)$ is not immediately obvious from the combinatorial definition of even and odd strings in \cite{AGL12}, but in fact, it is likely that with precise asymptotic constants it can be shown that $\ospt(N) < p(N)$ for all $N$.  It would be interesting to obtain a combinatorial proof of this as well.
\end{remark*}

Finally we prove parity results for $\ospt (N)$ that resemble those for $\spt (N)$.
\begin{theorem}
\label{T:ospt2}
If $N \in \N$, then $\ospt(N)$ is odd if and only if $24N - 1 = \ell^{4a+1} m^2$ for some prime $\ell \equiv 23 \pmod{24}$ and positive integers $a, m$ with $(\ell, m) = 1.$
\end{theorem}
\begin{remark*}
This result is an immediate corollary of a short proof that $\ospt(N) \equiv \spt(N) \pmod{2}$.  The parity condition in the theorem statement is therefore exactly that given for $\spt(N)$ in Theorem 1.3 of \cite{AGL12}.
\end{remark*}

The remainder of the paper is as follows.  In Section \ref{S:Symm} we define symmetrized positive crank and rank moments, give their generating functions, and describe how their study is sufficient to prove our main results.  Next, in Section \ref{S:Asymp} we use the Mittag-Leffler decomposition and Taylor's theorem in order to find the asymptotic behavior of the symmetrized positive moment generating functions.  In Section \ref{S:Circle} we apply Wright's variant of the Hardy-Ramanujan Circle Method to find the asymptotic behavior of the symmetrized positive moments, which proves our main results.  Section \ref{S:ospt} contains a detailed look at the case of the first moments, including the asymptotic behavior and parity of the ospt-function. Finally, one of our main technical tools is relegated to a brief Appendix, in which we apply a method described by Zagier to find asymptotic expansions for several series that arise during the course of the paper.

\section{Symmetrized moments}
\label{S:Symm}

Just as in Andrews' study of Durfee symbols  \cite{And07} and Garvan's work on higher spt-functions \cite{Gar11}, it is much more convenient to work primarily with {\it symmetrized} positive moments rather than with Definition \ref{D:MN+}.  Modifying Andrews' and Garvan's notation in the natural way, we set
\begin{align*}
\mu_r^+ (N) & := \sum_{m\in \N}\binom{m+ \left\lfloor \frac{r-1}2 \right\rfloor}{r} M(m,N) , \\
\eta_r^+ (N) & := \sum_{m\in \N} \binom{m+ \left\lfloor \frac{r-1}2 \right\rfloor}{r} N(m,N) .
\end{align*}

\subsection{Main results for symmetrized positive moments}
Our main asymptotic results in Theorem \ref{T:maintheorem2} are consequences of the following result for symmetrized moments, in which we write the leading terms using modified Bessel functions (following the standard notation $I_\nu(x)$). Using Bessel functions rather than exponentials makes it somewhat easier to compare the two moment functions, as the constants are simpler.  There is a small difference in some of the lower order terms that depends on the parity of $r$, although we always suppress the dependence on $r$ in defining
\begin{equation*}
\rho = \rho(r) := \begin{cases} 0 \qquad & \text{if } r \text{ is odd}, \\
\frac{1}{2} & \text{if } r \text{ is even}. \end{cases}
\end{equation*}

\begin{theorem}
\label{T:symmetrizedtheorem}
Suppose that $r \geq 2$.
As $N \to \infty$,
\begin{align*}
\mu_r^+ (N) &=   c_r N^{\frac{r}{2}-\frac34}\text{I}_{r-\frac32}\left(\pi\sqrt{\frac{2N}{3}}\right)
+ d_{1, r} N^{\frac{r}{2}-\frac54}\text{I}_{r-\frac52}\left(\pi\sqrt{\frac{2N}{3}}\right) + O\left(N^{\frac{r}{2}-\frac74} e^{\pi\sqrt{\frac{2N}{3}}}\right), \\
\eta_r^+ (N) &=   c_r N^{\frac{r}{2}-\frac34}\text{I}_{r-\frac32}\left(\pi\sqrt{\frac{2N}{3}}\right)
+ d_{3, r} N^{\frac{r}{2}-\frac54}\text{I}_{r-\frac52}\left(\pi\sqrt{\frac{2N}{3}}\right) + O\left(N^{\frac{r}{2}-\frac74} e^{\pi\sqrt{\frac{2N}{3}}}\right),
\end{align*}
where the constants $c_r$ and $d_{\ell, r}$ are given by
\begin{align*}
c_r &:= \zeta (r) \left( 1- 2^{1-r}\right) \frac{6^{\frac{r}2 - \frac{3}{4}}}{\sqrt{2} \pi^{r-1}},\\
d_{\ell, r} &:= - \frac{\pi}{24\sqrt{6}} \: c_r
- \frac{6^{\frac{r}{2} - \frac{5}{4}}}{\sqrt{2} \pi^{r-2}}
\left(\frac{\ell}{2} \zeta(r-2) \left(1 - 2^{3-r}\right) + \rho \zeta(r-1) \left(1 - 2^{1-r}\right) \right).
\end{align*}
\end{theorem}

\begin{proof}[Proof of Theorem \ref{T:maintheorem2} from Theorem \ref{T:symmetrizedtheorem}]
We first consider the case $r=1$, and note that Theorem \ref{T:osptAsym} is an equivalent restatement of Theorem \ref{T:maintheorem2} \ref{T:maintheorem2:M>N}.  Combined with \eqref{E:M1d} and \eqref{E:dlambda}, we can then also conclude Theorem \ref{T:maintheorem2} \ref{T:maintheorem2:MsimN}.

We thus assume that $r \geq 2$, and use Theorem \ref{T:symmetrizedtheorem}.  Observe (cf. Garvan's approach to even moments in \cite{Gar11}) that
\begin{align}
\label{E:MmuNeta}
M_r^+ (N) & = r! \mu_r^+ (N) + \sum_{\ell=0}^{r-1} a_\ell \: \mu_\ell^+ (N), \\
N_r^+ (N) & = r! \eta_r^+ (N) + \sum_{\ell=0}^{r-1} a_\ell \: \eta_\ell^+ (N), \notag
\end{align}
for certain integers $a_\ell$.  In order to obtain the asymptotic formulas as stated, we use the following well-known asymptotic formula for the modified Bessel functions as the argument $x \to \infty$ (which hold for any index $\nu$ (see (4.12.7) in \cite{AAR99}):
\begin{equation*}
I_\nu(x) = \frac{e^x}{\sqrt{2 \pi x}} + O\left(\frac{e^x}{x^{\frac32}}\right).
\end{equation*}
This immediately gives part \ref{T:maintheorem2:MsimN}.

For the difference of the crank and rank moments, we subtract the second line of \eqref{E:MmuNeta} from the first, obtaining
\begin{align*}
\label{E:MmuNeta}
M_r^+ (N) - N_r^+ (N) & = r! \left(\mu_r^+ (N) - \eta_r^+ (N)\right)
+ \sum_{\ell=0}^{r-1} a_\ell \: \left(\mu_\ell^+ (N) - \eta_\ell^+ (N)\right).
\end{align*}
Theorem \ref{T:symmetrizedtheorem} implies that every term in the sum is at most $O\left(N^{\frac{r}{2} - 2} e^{\pi \sqrt{\frac{2 N}{3}}}\right),$
and simplifying the first term gives part \ref{T:maintheorem2:M>N}.
\end{proof}

\subsection{Generating functions}
We now derive useful expressions for the generating functions of symmetrized positive moments, which we will use to analytically study the asymptotic behavior.  If $\ell, r \in\N$, define a ``false'' Appell-type sum by
\begin{equation}
\label{E:Slr}
S_{\ell, r} (q) := \sum_{n\geq 1} \frac{(-1)^{n+1} q^{\frac{\ell n^2}2 + \left(\frac{r}2 +\rho \right)n}}{\left( 1-q^n\right)^r}.
\end{equation}
We further set
$$
F_{\ell, r} (q) = \sum_{N \geq 0} a_{\ell, r} (N) q^N := \frac{1}{(q)_\infty} S_{\ell, r} (q).
$$
\begin{proposition}
Suppose that $r\in\N.$ Then we have
\begin{align*}
\sum_{N\geq 0} \mu_r^+ (N) q^N &=  F_{1,r} (q), \\
\sum_{N\geq 0} \eta_r^+ (N) q^N &=  F_{3,r} (q).
\end{align*}
\end{proposition}
\begin{remark*}
In other words, $\mu_r^+ (N) = a_{1, r}(N)$ and $\eta_r^+(N) = a_{3,r}(N).$
\end{remark*}
\noindent We do not provide a proof of these formulas, as they are entirely analogous to the proof for the case $r=1$ as found in Theorem 1 of \cite{AGL12} and the case of even $r$ found in Theorem 2 of \cite{And07}.  In order to write uniform formulas for all $r$, we use the fact that $r - \lfloor \frac{r-1}{2} \rfloor = \frac{r}{2} + \rho.$

\section{Asymptotic behavior of generating functions}
\label{S:Asymp}

In this section, we describe the asymptotic behavior of the generating functions $F_{\ell,r}(q)$ when $q$ is near an essential singularity on the unit circle.  We therefore adopt the notation commonly used with modular forms and write $q = e^{2 \pi i \tau}$, where $\tau = x + iy$ and $y > 0$.  The overall asymptotic behavior is largely controlled by the exponential singularities of $(q;q)_\infty^{-1}$, which can be easily understood through modular transformations.  The dominant pole is at $q = 1$, and the primary technical challenge is therefore to understand the asymptotic behavior of the sum $S_{\ell, r}(q)$ near this point.

\subsection{Bounds near the dominant pole}
\label{S:Asymp:dominant}
In order to identify the nature of the pole at $q = 1$ for $S_{\ell, r}(q)$, we apply the Mittag-Leffler partial fraction decomposition.
It is straightforward to verify that for $w\in\C$
\begin{equation}\label{ML}
\frac{e^{\pi iw}}{\left(1-e^{2\pi iw}\right)^r}
=\frac{1}{(-2\pi iw)^r}+\sum\limits_{0<j<r\atop{j\equiv r\pmod{2}}}\frac{\alpha_j}{(-2\pi iw)^j}
+\sum\limits_{0<j\leq r\atop{j\equiv r\pmod{2}}}\frac{\alpha_j}{(-2\pi i)^j}\sum_{m\geq 1}\left[\frac{1}{(w-m)^j}+\frac{1}{(w+m)^j}\right],
\end{equation}
where $\alpha_j$ are certain constants; note that $\alpha_r = 1$.
By summing this expansion over all $n$ (with $w=n\tau$), we obtain
\begin{align}
\label{E:SMittag}
S_{\ell, r}(q)=\sum_{n\geq 1}(-1)^{n+1}q^{\frac{\ell n^2}{2}+\rho n}
\Bigg[& \frac{1}{(-2\pi in\tau)^r}+\sum\limits_{0<j<r\atop{j\equiv r\pmod{2}}}\frac{\alpha_j}{(-2\pi in\tau)^j}\\
& \qquad +\sum\limits_{0<j\leq r\atop{j\equiv r\pmod{2}}}\frac{\alpha_j}{(-2 \pi i)^j} \sum_{m\geq 1}\left(\frac{1}{(n\tau-m)^j}+\frac{1}{(n\tau+m)^j}\right)\Bigg]. \notag
\end{align}

We now use this expression in order to determine the first terms in the asymptotic expansion of $S_{\ell, r}(q)$ in a certain analytic neighborhood of $q = 1$.
\begin{proposition}
\label{P:Slr}
Assume $r\geq 3$, $y=\frac{1}{2\sqrt{6N}}$, and $|x|\leq y$.  As $N \to \infty$, we have
\begin{equation*}
S_{\ell, r}(q)-\widetilde{c}_r(-2\pi i\tau)^{-r}-\widetilde{d}_{\ell, r}(-2\pi i\tau)^{-r+1}\ll N^{\frac{r}{2}-1}
\end{equation*}
with
\begin{align*}
\widetilde{c}_r&:=\zeta(r)\left(1-2^{1-r}\right),\\
\widetilde{d}_{\ell, r}&:= - \frac{\ell}{2}\zeta(r-2)\left(1-2^{3-r}\right)
- \rho\zeta(r-1)\left(1-2^{2-r}\right).
\end{align*}
\end{proposition}

\begin{proof}
We first consider the contribution from the first bracketed term in \eqref{E:SMittag}.  This reduces to the study of the functions
\begin{equation}\label{glj}
g_{\ell, j}(\tau):=\sum_{n\geq 1}(-1)^{n+1}n^{-j}q^{\frac{\ell n^2}{2}+\rho n}.
\end{equation}
If $j \geq 1$, then $g_{\ell, j}$ is convergent at $\tau = 0$, with value
\begin{align}
\label{E:g0}
g_{\ell, j}(0) = \sum_{n\geq 1}\frac{(-1)^{n+1}}{ n^{j}}& = \zeta(j)\left(1-2^{1-j}\right), \notag
\end{align}
where this expression is again to be interpreted as a limit when $j=1$.
Note further that if $j\geq 2$, then $g_{\ell, j}$ is absolutely (and uniformly) convergent for all $| q | \leq 1$, since
\begin{equation*}
\Big|g_{\ell, j}(\tau)\Big| \leq \sum_{n\geq 1}\frac{1}{n^{j}}=\zeta(j).
\end{equation*}

We will apply Taylor's Theorem to obtain lower-order asymptotic terms in $g_{\ell, r}$, and make use of the fact that the resulting derivatives can be expressed recursively using
\begin{align*}
\frac{1}{2\pi i}\frac{\partial}{\partial\tau}g_{\ell, j}(\tau)&=\frac{\ell}{2}g_{\ell, j-2}(\tau)+\rho g_{\ell, j-1}(\tau),\\
\left(\frac{1}{2\pi i}\right)^2\frac{\partial^2}{\partial\tau^2}g_{\ell, j}(\tau)&=\frac{\ell^2}{4}g_{\ell, j-4}(\tau)+\ell \rho g_{\ell, j-3}(\tau)
+\rho^2 g_{\ell, j-2}(\tau). \notag
\end{align*}
Taylor's Theorem can now be directly applied for $j\geq 6$ (as all terms are uniformly bounded), giving
\begin{equation}
\label{E:gTaylor}
g_{\ell, j}(\tau) - g_{\ell, j}(0) - g_{\ell, j}'(0)\tau
\ll |\tau|^2\sup_{w\in\mathbb{H}}\Big|g_{\ell, j}''(w)\Big|\ll |\tau|^2\ll N^{-1}.
\end{equation}
The main terms in the proposition statement are given by the first two terms in \eqref{E:gTaylor}, so the remaining task in the proof is to show that the other summands \eqref{E:SMittag} are of smaller order.  Additionally, we need to individually address the small values of $j$.

The second bracketed term in \eqref{E:SMittag} is a sum on $j$, and on these terms we use a simple uniform bound.  In particular,
\begin{align*}
g_{\ell, j}(\tau) \ll \sum_{n\geq 1} \frac{e^{- C n^2y}}{n^j} \ll \begin{cases} 1 \qquad & \text{if } j \geq 2, \\
1+\frac{1}{\sqrt{y}} \ll N^\frac14 \qquad & \text{if } j = 1. \end{cases}
\end{align*}
for some constant $C$.  In the case $j = 1$, we bound the (monotonic) sum by an integral, and the statement follows by the standard Gaussian evaluation.

For the final bracketed term in \eqref{E:SMittag}, a direct calculation shows that for $m\in\Z\setminus\{0\}$ (and values of $\tau$ subject to the constraint $|x| \leq y$)
\begin{equation*}
\frac{1}{n\tau+m} \ll \frac{1}{m}.
\end{equation*}

Thus for $j>1$ we have the bound
\begin{equation}
\label{E:thetabound}
\sum_{n\geq 1}(-1)^n q^{\frac{\ell n^2}{2}+\rho n}\sum_{m\geq 1}\left(\frac{1}{(n\tau-m)^j}+\frac{1}{(n\tau+m)^j}\right)
\ll\sum_{m\geq 1} m^{-j}\sum_{n\geq 1} e^{-\pi \ell n^2y}\ll y^{-\frac12}\ll N^{\frac14}.
\end{equation}
The final summation bound follows again through a comparison with a Gaussian integral.  Alternatively, it can also be shown using the modular inversion formula for the theta function; it also follows from Zagier's formula for asymptotic expansions (cf. Proposition \ref{P:Zag}; also found as Example 2 in Section 4 of \cite{Zag06}).
Moreover, for $j=1$ we have
\begin{equation*}
\frac{1}{n\tau-m}+\frac{1}{n\tau+m} \ll \frac{n\tau}{m^2}.
\end{equation*}
Thus the corresponding contribution can be bounded by
\begin{equation}
\label{E:j=1Int}
|\tau| \sum_{n\geq 1}n e^{-\pi n^2y} \ll N^{-\frac12}\frac{1}{y} \ll 1,
\end{equation}
where the sum is bounded using Proposition \ref{P:j=1Sum} in the Appendix.  This completes the proof for $r \geq 6$.

For the remaining cases $3 \leq r \leq 5$, absolute bounds are not sufficient to show that the Taylor series in \eqref{E:gTaylor} is valid (which requires that the bounds are uniform), so we instead argue directly that the function $g_{\ell, j}$ is uniformly bounded for $j=-1,0,1.$  The details are somewhat involved and are therefore found in Appendix \ref{S:append}, where we follow Zagier's use of a technical tool related to the Mellin transform to study the asymptotic expansions of series such as $g_{\ell, j}$.  See Proposition \ref{P:glj} for this final part of the proof.
\end{proof}

\begin{corollary}
\label{C:Flrdominant}
Assume that $r \geq 3$, $y=\frac{1}{2\sqrt{6N}}$, and $|x|\leq y$.  As $N \to \infty$, we have
\begin{equation*}
F_{\ell, r}(q)
- c_r^\ast(-2 \pi i\tau)^{\frac12-r} e^{\frac{\pi i}{12\tau}}
-d_{\ell, r}^\ast(-2 \pi i\tau)^{\frac32-r}e^{\frac{\pi i}{12\tau}}\ll N^{\frac{r}{2}-\frac54}e^{\pi\sqrt{\frac{N}{6}}},
\end{equation*}
where
\begin{align*}
c_r^\ast & := \frac{\widetilde{c}_r}{\sqrt{2 \pi}}, \qquad \text{ and } \quad
d_{\ell, r}^\ast := \frac{1}{\sqrt{2\pi}}\left(-\frac{\widetilde{c}_r}{24}+\widetilde{d}_{\ell, r}\right).
\end{align*}
\end{corollary}

\begin{proof}
Recall the modular inversion formula for Dedekind's eta function (Theorem 3.1 in \cite{Apo90}), which states that
\begin{equation}\label{inveta}
\eta\left(-\frac1\tau\right)=\sqrt{-i\tau}\eta(\tau).
\end{equation}
We use this to obtain the expansion
\begin{align}
\label{E:qinfty}
\frac{1}{(q)_\infty}= \frac{q^{\frac{1}{24}}\sqrt{-i\tau}}{\eta\left(-\frac1\tau\right)}
& =\sqrt{-i\tau} e^{\frac{2\pi i}{24}\left(\tau+\frac1\tau\right)}
\Big(1+O\Big(e^{-2\pi \sqrt{6N}}\Big)\Big) \notag \\
& = \sqrt{-i\tau} e^{\frac{\pi i}{12\tau}}\left(1+\frac{2\pi i\tau}{24}+O\left(N^{-1} \right)\right).
\end{align}

Along with Proposition \ref{P:Slr}, this implies the claimed expansion, as
\begin{multline*}
\frac{1}{(q)_\infty} S_{\ell, r}(q)=\sqrt{-i\tau}e^{\frac{\pi i}{12\tau}}\left(1+\frac{2\pi i\tau}{24}+O\left(N^{-1}\right)\right)
\Big(\widetilde{c}_r (-2\pi i\tau)^{-r} + \widetilde{d}_{\ell, r}(-2\pi i\tau)^{-r+1} + O\left(N^{\frac{r}{2}-1}\right)\Big)\\
=\widetilde{c}_r\sqrt{-i\tau}e^{\frac{\pi i}{12\tau}}(-2\pi i\tau)^{-r}
+\left(-\frac{\widetilde{c}_r}{24}+\widetilde{d}_{\ell, r}\right)\sqrt{-i\tau} e^{\frac{\pi i}{12\tau}}
(-2\pi i\tau)^{-r+1}
+O\left(N^{\frac{r}{2}-\frac54}e^{\pi\sqrt{\frac{N}{6}}}\right).
\end{multline*}
\end{proof}

\subsection{Bounds away from the dominant pole}
\label{S:Asymp:away}
We next consider the behavior of $F_{\ell, r}$ when $q$ is not near $1$.  First, we provide a simple uniform bound for $S_{\ell, r}$.
\begin{proposition}
\label{P:Slraway}
If $y=\frac{1}{2\sqrt{6N}}$, then we have for $r\geq 1$
\begin{equation*}
\Big|S_{\ell, r}(q)\Big|\ll N^{\frac{r}{2}+\frac14}.
\end{equation*}
\end{proposition}
\begin{proof}
Bounding each term in \eqref{E:Slr} absolutely, we find that
\begin{equation*}
\Big|S_{\ell, r}(q)\Big|\leq \frac{1}{(1-|q|)^r}\sum_{n\geq 1}|q|^{\frac{n^2}{2}}\ll N^{\frac{r}{2}}N^{\frac14}=N^{\frac{r}{2}+\frac14},
\end{equation*}
where the final bound follows as in \eqref{E:thetabound}
\end{proof}

This leads to a uniform bound in the region away from $q \sim 1$ that is exponentially smaller than the asymptotic formulas from Section \ref{S:Asymp:dominant}.
\begin{corollary}
\label{C:Flraway}
If $y=\frac{1}{2\sqrt{6N}}$ and $y\leq|x|\leq \frac12$, then we have for $r\geq 1$
\begin{equation*}
\left\lvert F_{\ell, r}(q)\right\rvert \ll N^{\frac{r}{2}+\frac14} e^{\frac{\pi}{2}\sqrt{\frac{N}{6}}}.
\end{equation*}
\end{corollary}
\begin{proof}
We again use the inversion formula \eqref{inveta} to obtain the bound

\begin{equation*}
\frac{1}{|(q)_\infty|}\ll|\tau|^{\frac12}e^{\frac{\pi y}{12\left(x^2+y^2\right)}}\ll e^{\frac{\pi}{24y}}\ll e^{\frac{\pi}{2}\sqrt{\frac{N}{6}}}.
\end{equation*}
Combined with Proposition \ref{P:Slraway}, this gives the claimed expression.
\end{proof}

\section{The Circle Method}
\label{S:Circle}

In this section we apply Wright's variant of the Hardy-Ramanujan Circle Method and complete the proof of our asymptotic expressions for the positive crank and rank moments.  His approach differs from Hardy-Ramanujan and Rademacher's mainly in the choice of certain parameters (which were reflected in the ranges we considered in Section \ref{S:Asymp}), so the initial setup begins as usual.
Cauchy's Theorem gives an integral representation for the coefficients of $F_{\ell,r}$, namely

\begin{equation}
\label{E:alr}
a_{\ell, r} (N) = \frac1{2\pi i} \int_{\mathcal{C}} \frac{F_{\ell, r} (q)}{q^{N+1}} dq = \int_{-\frac12}^{\frac12} F_{\ell, r} \left( e^{-\frac{\pi}{\sqrt{6N}} + 2\pi i x} \right) e^{\pi\sqrt{\frac{N}6} - 2\pi i Nx} dx,
\end{equation}
where the contour is the counterclockwise traversal of the circle $\mathcal{C} :=\left\{ |q|=e^{-\frac{\pi}{\sqrt{6N}}} \right\}$.  We separate \eqref{E:alr} into two ranges, writing $a_{\ell,r}(N) = \text{I}' + \text{I}'',$ with
\begin{align*}
 \text{I}'  &:=  \int_{|x| \leq \frac1{2\sqrt{6N}}} F_{\ell, r} \left( e^{-\frac{\pi}{\sqrt{6N}} + 2\pi i x} \right) e^{\pi\sqrt{\frac{N}6} - 2\pi i Nx} dx, \\
 \text{I}''  &:=  \int_{\frac1{2\sqrt{6N}} \leq |x| \leq \frac12} F_{\ell, r} \left( e^{-\frac{\pi}{\sqrt{6N}} + 2\pi i x} \right) e^{\pi\sqrt{\frac{N}6} - 2\pi i Nx} dx.
\end{align*}
The first integral provides the main asymptotic contribution, while the second is of exponentially lower order and will be absorbed into the error term.

\subsection{Main arc}
\label{S:Circle:Main}
We now show that the main asymptotic terms stated in Theorem \ref{T:symmetrizedtheorem} arise from $\text{I}'$, and begin by rewriting the integral as
\begin{equation}
\label{E:I'}
\text{I}' = \frac{1}{2 \sqrt{6N}} \int_{-1}^1 F_{\ell,r}\left(e^{\frac{\pi}{\sqrt{6N}} \left(-1 + iu\right)}\right)
e^{\pi \sqrt{\frac{N}{6}} \left(1 - iu\right)} \: du.
\end{equation}
In other words, we write $\tau = \frac{1}{2 \sqrt{6N}} (u + i)$ with $|u| \leq 1.$
Our immediate goal is to rewrite $I'$ equivalently in terms of Bessel functions, up to an allowable error.

\begin{proposition}
\label{P:I'}
Assume that $r\geq 3$.
As $N \to \infty$,
\begin{equation*}
\mathrm{I}' - c_r N^{\frac{r}{2}-\frac34}\text{I}_{r-\frac32}\left(\pi\sqrt{\frac{2N}{3}}\right)
- d_{\ell, r} N^{\frac{r}{2}-\frac54}\text{I}_{r-\frac52}\left(\pi\sqrt{\frac{2N}{3}}\right)
\ll N^{\frac{r}{2}-\frac74}e^{\pi\sqrt{\frac{2N}{3}}}.
\end{equation*}
\end{proposition}

Before proving this result, we first relate the integration terms that directly appear in the evaluation of $\text{I}'$ to modified Bessel functions.  For $s \in \R$, we adopt Wright's definition \cite{Wri41} of an auxiliary function
\begin{equation*}
P_s:=\frac{1}{2\pi i}\int_{1-i}^{1+i} v^s e^{\pi\sqrt{\frac{N}{6}}\left(v+\frac1v\right)}dv.
\end{equation*}
The next lemma relates $P_s$ to an appropriate modified Bessel function.
\begin{lemma}
\label{L:PI}
As $N \to \infty$,
\begin{equation*}
P_s- I_{-s-1}\left(\pi\sqrt{\frac{2N}{3}}\right) \ll e^{\frac{3\pi}{2}\sqrt{\frac{N}{6}}}.
\end{equation*}
\end{lemma}
\begin{proof}
The loop integral representation for the modified Bessel functions \cite{Arf85} states that
\begin{equation*}
I_{-s-1}\left(\pi\sqrt{\frac{2N}{3}}\right)=\frac{1}{2\pi i}\int_\Gamma e^{\pi\sqrt{\frac{N}{6}}\left(t+\frac1t\right)}t^s dt,
\end{equation*}
where $\Gamma$ is Hankel's standard contour that begins in the lower half-plane at $-\infty$, proceeds counterclockwise around the origin, and returns to $-\infty$ in the upper-half plane.
In our setting, we set $\Gamma$ to be the piecewise linear path consisting of the segments
\[
\gamma_4: \left(-\infty - \frac{i}{2}, -1 - \frac{i}{2}\right),\quad \gamma_3: \left(-1 - \frac{i}{2}, -1 - i\right), \quad \gamma_2: \left(- 1 - i, 1 - i \right), \text{ and }\gamma_1: \left(1 - i, 1+i\right),
\]
which are then followed by the corresponding conjugate mirror images $\gamma'_2, \gamma'_3,$ and $\gamma'_4$.  Since $P_s = \int_{\gamma_1},$ it is sufficient to show that all other segments are bounded as claimed.

This contour is very similar to that used in Lemma 17 of Wright \cite{Wri41}, though he did not use an infinite contour, since he only considered integral $s$ (and therefore had no branch cuts).  He directly showed (equations (5.53) and (5.55) of \cite{Wri41}) that
\begin{equation*}
\int_{\gamma_2} \ll e^{\frac{3\pi}{2}\sqrt{\frac{N}{6}}} \qquad \text{and } \quad
\int_{\gamma_3} \ll e^{-\pi \sqrt{\frac{N}{6}}},
\end{equation*}
so we need only bound the integral along $\gamma_4$.  On this segment,
$\text{Re}\left(\frac{1}{t}\right) < 0$, so
\begin{align*}
\int_{\gamma_4} \ll \int_1^\infty e^{-\pi\sqrt{\frac{N}{6}}t} \left\lvert-t+\frac{i}{2}\right\rvert^s dt
\ll N^{-\frac{1}{2}} e^{-\pi\sqrt{\frac{N}{6}}},
\end{align*}
where the final asymptotic inequality follows from a simple bound for incomplete Gamma function.
\end{proof}

\begin{proof}[Proof of Proposition \ref{P:I'}]

By Corollary \ref{C:Flrdominant} and \eqref{E:I'} we have
\begin{align*}
\text{I}' = \frac{1}{2\sqrt{6N}}
\int_{-1}^1 & \left[ c_r^\ast \left( \frac{\pi(1-iu)}{\sqrt{6N}}\right)^{\frac12-r}
e^{\pi\sqrt{\frac{N}{6}}\left(\frac{1}{iu-1}+(1-iu)\right)}\right. \\
& \qquad \left. - d_{\ell,r}^\ast \left( \frac{\pi(1-iu)}{\sqrt{6N}}\right)^{\frac{3}{2}-r}
e^{\pi\sqrt{\frac{N}{6}}\left(\frac{1}{iu-1}+(1-iu)\right)}
+ O\left( N^{\frac{r}{2}-\frac54} e^{2 \pi\sqrt{\frac{N}{6}}} \right)
 \right] du,
\end{align*}
and the error term gives the (absolute) error bound as claimed in the statement.  For the two asymptotic terms we calculate the general identity
\begin{align*}
\frac{1}{2 \sqrt{6N}} \int_{-1}^1 \left(\frac{\pi(1-iu)}{\sqrt{6N}}\right)^s
e^{\pi\sqrt{\frac{N}{6}}\left(\frac{1}{iu-1}+(1-iu)\right)} du
& = \frac{i}{2 \sqrt{6N}} \int_{1+i}^{1-i} \left(\frac{\pi v}{\sqrt{6N}}\right)^s
e^{\pi \sqrt{\frac{N}{6}} \left(\frac{1}{v} + v\right)} \: dv \\
& = \left( \frac{\pi}{\sqrt{6N}} \right)^{s+1} P_s,
\end{align*}
where we have made the change of variables $u = i(v-1).$
Along with Lemma \ref{L:PI}, this completes the proof.
\end{proof}

\subsection{Error arc}
\label{S:Circle:Minor}
We now turn to the integral $\text{I}''$ and show that it is exponentially smaller than the main asymptotic terms in Theorem \ref{T:symmetrizedtheorem}.
\begin{proposition}
As $N \to \infty,$
\begin{equation*}
\text{I}''\ll N^{\frac{r}{2}+\frac14} e^{\frac{3\pi}{2}\sqrt{\frac{N}{6}}}.
\end{equation*}
\end{proposition}

\begin{proof}
Using Corollary \ref{C:Flraway}, we have the simple absolute bound
\begin{equation*}
|\text{I}''|\leq \int_{\frac{1}{2\sqrt{6N}}\leq |x|\leq \frac12}\left\lvert F_{\ell, r}\Big(e^{-\frac{\pi}{\sqrt{6N}}+2\pi ix}\Big) e^{\pi \sqrt{\frac{N}{6}}-2\pi iNx}\right\rvert dx
\leq e^{\frac{\pi\sqrt{N}}{6}}N^{\frac{r}{2}+\frac14} e^{\frac{\pi}{2}\sqrt{\frac{N}{6}}}.
\end{equation*}
\end{proof}

We have thus proven the following asymptotic result that gives Theorem \ref{T:symmetrizedtheorem} when $\ell = 1$ and $3$.
\begin{corollary}
Suppose that $r \geq 3.$  As $N \to \infty$,
\begin{align*}
a_{\ell,r}(N) - c_r N^{\frac{r}{2}-\frac34}\text{I}_{r-\frac32}\left(2\pi\sqrt{\frac{N}{6}}\right)
- d_{\ell, r} N^{\frac{r}{2}-\frac54}\text{I}_{r-\frac52}\left(2\pi\sqrt{\frac{N}{6}}\right)
\ll N^{\frac{r}{2}-\frac74}e^{2\pi\sqrt{\frac{N}{6}}}.
\end{align*}
\end{corollary}
This completes the proof of Theorem \ref{T:symmetrizedtheorem} for $r \geq 3.$  Furthermore, the case $r = 2$ of Theorem \ref{T:symmetrizedtheorem} is found in the proof of Theorem 2.1 of \cite{BM09}.

\section{Properties of the ospt-function}
\label{S:ospt}

In this section we prove Theorems \ref{T:osptAsym} and \ref{T:ospt2}.  We begin by proving a combinatorial inequality for the ospt-function, and then proceed by describing the relationship between the ospt-function and previous analytic and arithmetic results.

\subsection{Combinatorial results}
\label{S:ospt:Comb}

We first define several partition statistics.  Suppose that $\lambda$ is a partition.  A {\it run of length $L$} is a sequence of consecutive integers such that each of $m, m+1, \dots, m+L-1$ occur as parts in $\lambda$, and $m +L$ does not (however, $m-1$ may occur in $\lambda$.).  Following Andrews, Chan, and Kim \cite{ACK12}, we further define an {\it odd string} to be a run of length $L$ that begins at some odd positive integer $2k +1$, such that $2k+1$ occurs only once in $\lambda$ and $L \geq 2k+1$.  Similarly, an {\it even string} is a run of length $L$ that begins at some positive even integer $2k$, with the additional conditions that $2k -1$ does not occur as a part, and that $L \geq 2k-1$ is odd.
The total number of even or odd strings contained in $\lambda$ is then denoted by $\ST(\lambda)$ (note that a single partition may contain multiple even and/or odd strings).

In a remark following Corollary 8 of \cite{ACK12}, Andrews, Chan, and Kim showed that $\ospt(N)$ is always positive.  In fact, we now prove that the ospt-function is weakly increasing.
\begin{lemma}
\label{L:osptInc}
For $N \geq 1$,
\begin{equation*}
\ospt(N + 1) \geq \ospt(N).
\end{equation*}
\end{lemma}
\begin{proof}
Suppose that $\lambda$ is a partition of size $N$ that contains at least one even or odd string.  We describe a simple injection that sends $\lambda$ to a partition $\lambda'$ of $N+1$ such that $\ST(\lambda) \leq \ST(\lambda')$; this is sufficient to prove the claim by \eqref{E:osptST}.

If $\lambda$ does not contain any $1$s, then we define $\lambda'$ by adding a $1$ to $\lambda$.  There is now an odd string in $\lambda'$ beginning with $1$.  If there was an even string in $\lambda$ that began with $2$, then this odd string in $\lambda'$ serves as its replacement in the total enumeration, and the odd string is otherwise an additional string; in either case we have $\ST(\lambda) \leq \ST(\lambda')$.

If $\lambda$ contains exactly one $1$, then it contains an odd string beginning with $1$.  We define $\lambda'$ by removing the $1$ and adding a $2$ in its place, which deletes the odd string.  However, this then creates an even string beginning with $2$, since even if there were already parts of size $2$ in $\lambda$, there were no even strings due to the presence of the part $1$.  Furthermore, this procedure does not affect any other strings in $\lambda$, so  again $\ST(\lambda) = \ST(\lambda')$.

Finally, if $\lambda$ contains two or more $1$s, we construct $\lambda'$ by adding an additional $1$.  Since $\lambda$ does not contain any odd or even strings beginning with $1$ or $2$, the construction of $\lambda'$ does not affect the total string count, and thus $\ST(\lambda) = \ST(\lambda').$
\end{proof}

\subsection{Asymptotic behavior}
\label{S:ospt:Asymptotic}

We recall the generating function for the ospt-function from Section 3 of \cite{ACK12}, namely
\begin{equation}
\label{E:osptGen}
\mathcal{O}(q) := \sum_{N\geq 0} \text{ospt} (N) q^N = \frac1{(q)_\infty} \sum_{n\geq 1} \frac{(-1)^{n+1} q^{\frac{n^2+n}2} \left(1-q^{n^2} \right)}{1-q^n}.
\end{equation}
For convenience, we define the following notation for the hypergeometric sum portion:
$$
T(q) := \sum_{n\geq 1} \frac{(-1)^{n+1} q^{\frac{n^2+n}2} \left(1-q^{n^2} \right)}{1-q^n}.
$$

The principle estimate needed to conclude Theorem \ref{T:osptAsym} is the following asymptotic result for $T(q).$
\begin{proposition}
\label{P:Tq}
We have
$$
\lim_{y\rightarrow 0^+} T\left(e^{-y} \right) = \frac14.
$$
\end{proposition}
\begin{proof}
As before, we use the Mittag-Leffler decomposition \eqref{ML}
\begin{equation}
\label{E:Mittag2}
\frac{e^{-\frac{y}{2}}}{1 - e^{-y}} = \frac1{y} + i \sum_{m\geq 1} \left(\frac1{iy-2\pi m} + \frac1{iy+2\pi m}\right).
\end{equation}
The contribution of the first summand of \eqref{E:Mittag2} to $T\left(e^{-y}\right)$ is described by the asymptotic expansion of $y^{-1} g(y)$, where
\begin{equation}
\label{E:gy}
g(y) := \sum_{n\geq 1} \frac{(-1)^{n+1} e^{-\frac{n^2y}{2}} \left(1-e^{-n^2y} \right)}{n}.
\end{equation}
Proposition \ref{P:gy} implies that
\begin{equation*}
y^{-1} g(y) \sim \frac{1}{4},
\end{equation*}
so the remaining task is to show that the other terms from \eqref{E:Mittag2} do not make any contribution to the asymptotic main term of $T\left(e^{-y}\right)$.
By simplifying and regrouping, we find that the remaining sum can be bounded, up to asymptotic constants, by
\begin{equation*}
y \sum_{n\geq 1} e^{-n^2y} \left(1-e^{-2n^2 y} \right) n \sum_{m\geq 1} \frac1{n^2y^2 + m^2 },
\end{equation*}
which evaluates directly to $0$ when $y \to 0^+.$
\end{proof}

Ingham's Tauberian theorem relates the asymptotic behavior of such a series to its coefficients. The following result is a special case of Theorem 1 in \cite{Ing41} (see also Theorem 1').
\begin{theorem}[Ingham]\label{T:Ingham}
Let $f(z)=\sum_{N\geq 0}a(N) z^N$ be a power series with real nonnegative
coefficients and radius of convergence equal to $1$. If there exist $A>0$,
$\lambda, \alpha\in\R$ such that
\[
f(z)\sim
\lambda\left(-\log (z)\right)^\alpha
\exp\left(\frac{-A}{\log (z)}\right)
\]
as $z\to 1^-$, then
\[
\sum_{m=0}^N a(m)\sim\frac{\lambda}{2\sqrt{\pi}}\,
\frac{A^{\frac{\alpha}{2}-\frac14}}{N^{\frac{\alpha}{2}+\frac14}}\,
\exp\left(2\sqrt{AN}\right)
\]
as $N\to\infty$.
\end{theorem}

We are now able to prove the asymptotic behavior of $\ospt(N)$.
\begin{proof}[Proof of Theorem \ref{T:osptAsym}]
We use Inham's Tauberian theorem in order to conclude the asymptotic behavior of $\ospt(N).$  In particular, Lemma \ref{L:osptInc} implies that Theorem \ref{T:Ingham} may be applied to
\begin{equation*}
(1-q) \mathcal{O}(q) =1+ \sum_{N \geq 1} \left(\ospt(N) - \ospt(N-1)\right) q^N.
\end{equation*}
Combined with Proposition \ref{P:Tq} and \eqref{pasymp}, this implies that
\begin{equation*}
1+\sum_{m = 1}^N \left(\ospt(N) - \ospt(N-1)\right) = \ospt(N) \sim \frac{1}{16\sqrt{3}N} e^{\pi\sqrt{\frac{2N}{3}}}\sim\frac{1}{4} p(N),
\end{equation*}
which completes the proof.
\end{proof}

\subsection{A simple congruence modulo $2$}
\label{S:ospt:mod2}

Theorem \ref{T:ospt2} is an immediate corollary of Andrews, Garvan and Liang's description of the parity of $\spt(N)$ (see Theorem 1.3 of \cite{AGL12}) and the following result.
\begin{lemma}
For all $N \in \N$, we have
\begin{equation*}
\ospt(N) \equiv \spt(N) \pmod{2}.
\end{equation*}
\end{lemma}
\begin{proof}
Recalling \eqref{E:MN2k} and using the fact that $a^2 \equiv a \pmod{2}$ for any integer, we find that
\begin{equation*}
M_2^+(N)=\sum_{m\geq 1}  m^2 M(m, N)\equiv \sum_{m\geq 1}m M(m, N) =M_1^+(N)\pmod{2},
\end{equation*}
with a similar statement for the second rank moments.  Thus
\begin{equation*}
\spt(N) = M_2^+(N)-N_2^+(N)\equiv M_1^+(N)-N_1^+(N) \equiv \ospt(N)\pmod{2}.
\end{equation*}
\end{proof}

\appendix
\section{Asymptotic expansions of infinite sums}
\label{S:append}

In this section we present detailed proofs for the asymptotic behavior of several summations that appeared in Sections \ref{S:Asymp} and \ref{S:ospt}.  Our chief technical tool comes from Zagier's treatment of asymptotic expansions for series found in Section 4 of \cite{Zag06}.
In particular, suppose that a smooth function $f : (0, \infty) \to \C$ has an asymptotic power series expansion around $0$, which means that for any $S \geq 0$,
\begin{equation}
\label{E:ft}
f(t) = \sum_{n = 0}^S b_n t^n + O\left(t^{S+1}\right)
\end{equation}
as $t \to 0^+.$  For any $a > 0$, Zagier considered the summation
\begin{equation}
\label{E:gt}
g(t) := \sum_{m \geq 0} f((m+a)t),
\end{equation}
and showed that its asymptotic behavior can be simply described in terms of the coefficients of the expansion \eqref{E:ft}.  A function $f$ is said to be of {\it rapid decay} at infinity if $\int_{\ell}^\infty |f(u)| du$ converges for some $\ell > 0$.  The following result is stated as the first generalization of Proposition 3 in \cite{Zag06}.
\begin{proposition}
\label{P:Zag}
Suppose that $f$ has the asymptotic expansion \eqref{E:ft} and that $f$ and all of its derivatives are of rapid decay at infinity.  Suppose further that $I_f := \int_0^\infty f(u) du$ converges.  The function $g$ as defined in \eqref{E:gt} then has the asymptotic expansion
\begin{equation*}
g(t) = \frac{I_f}{t} - \sum_{n = 0}^S b_n \frac{B_{n+1}(a)}{n+1} t^n,
\end{equation*}
where $B_{n}(x)$ is the $n$-th Bernoulli polynomial, defined by
$\frac{t e^{xt}}{e^t - 1} = \sum_{n \geq 0} B_n(x) \frac{t^n}{n!}.$
\end{proposition}
\begin{remark*}
This statement corrects a sign error in \cite{Zag06}, where the sum is added rather than subtracted.
\end{remark*}

We first apply this result to prove a bound needed in the case $j=1$ in the proof of Proposition \ref{P:Slr}.
\begin{proposition}
\label{P:j=1Sum}
As $y \rightarrow 0^+$,
\begin{equation*}
\sum_{n \geq 1} n e^{-\pi n^2 y} \ll \frac{1}{y}.
\end{equation*}
\end{proposition}
\begin{proof}
We let $t := \sqrt{\pi y}$, and rewrite the sum as
\begin{equation}
\label{E:fsum}
\frac{1}{t} \sum_{n \geq 1} f(nt),
\end{equation}
where $f(t) := te^{-t^2}.$  Proposition \ref{P:Zag} then implies that
\begin{equation*}
\sum_{n \geq 1} f(nt) = \frac{I_f}{t} + O(1),
\end{equation*}
with the integral evaluating to $I_f = \frac{1}{2}.$  This implies the proposition statement.
\end{proof}

We next use the technique to complete the proof of Proposition \ref{P:Slr} in the cases of small $r$.
Recall the definition of $g_{\ell, j}$ may be found in \eqref{glj}.
\begin{proposition}
\label{P:glj}
Suppose that $y = \frac{1}{2 \sqrt{6N}}$ and $|x| \leq y.$  If $j = -1, 0,$ or $1$, then
\begin{equation*}
g_{\ell, j}(\tau) \ll 1.
\end{equation*}
\end{proposition}

\begin{proof}
We start with the case $j=1$, and first assume that $\ell$ is odd (so $\rho = 0$).
By definition,
\begin{equation}
\label{E:gell1}
g_{\ell, 1}  (\tau) = \sum_{n\geq 1} \left( \frac{e^{\pi i \ell (2n-1)^2\tau }}{2n-1} - \frac{e^{\pi i \ell (2n)^2\tau }}{2n} \right).
\end{equation}

We show that the real and imaginary parts are both individually bounded, focusing only on the real part, as the imaginary part is treated identically.  Euler's identity implies that
\begin{equation}
\label{E:g1Re}
\text{Re} \left( g_{\ell,1} (\tau) \right) =  \sum_{n\geq 1} \left( \frac{e^{ -  \pi \ell (2n-1)^2y} \cos \left( \pi\ell (2n-1)^2 x\right)}{2n-1} -\frac{e^{ -  \pi \ell (2n)^2y} \cos \left(\pi\ell (2n)^2  x\right)}{2n}\right).
\end{equation}
In order to combine these odd and even terms, we begin by writing
$$
	\frac1{2n} = \frac{1}{2n-1} + O\left(\frac{1}{n^2}\right).
$$
If this big-O term is inserted back into \eqref{E:g1Re}, the sum is again absolutely and uniformly convergent, so we may discard this error term without affecting the overall convergence.

We use the following trivial bounds in our estimates:
\begin{alignat*}{2}
|\cos (x+a) - \cos (x) | & \leq \min \left\{ |a|, 2 \right\}  && \qquad x,a,\in\R \\
|1-e^{-x} | & \leq \min \left\{x, 1 \right\} && \qquad x\geq 0.
\end{alignat*}
These bounds may be applied once the even terms in \eqref{E:g1Re} are rewritten in the following way:
\begin{multline}
\label{E:evenRewrite}
 \frac{e^{ -  \pi \ell (2n)^2y} \cos \left(\pi\ell (2n)^2 x \right)}{2n-1}
 \quad =  \frac{e^{ -  \pi \ell (2n-1)^2y} }{2n-1}
\bigg( \left( e^{ -\pi \ell(4n-1)  y } -1 \right) \cos \left( \pi\ell (2n)^2x\right)  \\ + \Big(-  \cos \left( \pi\ell (2n-1)^2x\right) +  \cos \left( \pi\ell  (2n)^2x\right)  \Big)
 - \cos \left( \pi\ell (2n-1)^2x\right) \bigg).
\end{multline}

The contribution of the first pair of terms in \eqref{E:evenRewrite} to the sum in \eqref{E:g1Re} is then asymptotically bounded (up to a constant) by
\begin{align}
\label{E:even1}
\sum_{n\geq 1} \frac{e^{-\pi\ell (2n-1)^2y }}{2n-1} \min \left\{y (4n-1), 1 \right\}
& \ll y \sum_{1\leq n\leq y^{-1}} e^{-n^2 y} + \sum_{n\geq y^{-1}} \frac{e^{-n^2y}}{n} \\
& \ll \sqrt{y} + y \sum_{n \geq y^{-1}} e^{-n^2 y} \ll \sqrt{y}, \notag
\end{align}
where the first sum was bounded using a comparison with a Gaussian integral and the second by a comparison with the incomplete gamma function.
The second summand from \eqref{E:evenRewrite} also gives (noting that $|x|<y$)
\begin{equation}
\label{E:gell1Bound}
\sum_{n\geq 1} \frac{e^{-\pi\ell(2n-1)^2  y}}{2n-1} \min \left\{ 1,  |x| (2n-1) \right\} \ll \sum_{n\geq 1} \frac{e^{-\pi\ell(2n-1)^2  y}}{2n-1} \min \left\{1, ny \right\} \ll \sqrt{y}.
\end{equation}
Since the third term in \eqref{E:evenRewrite} cancels the odd terms in \eqref{E:g1Re}, the real part is bounded overall, as claimed.

The proof is analogous if $\ell$ is even, with the only difference being minor shifts in all of the exponents.  In particular, in this case
\begin{equation*}
g_{\ell, 1}(\tau) = q^{-\frac{1}{8\ell}} \sum_{n \geq 1} \frac{(-1)^{n+1} q^{\frac{\ell}{2} \left(n + \frac{1}{2\ell}\right)^2}}{n},
\end{equation*}
so all occurrences of $(2n-1)^2$ and $(2n)^2$ from \eqref{E:gell1} through \eqref{E:gell1Bound} are simply replaced by $\left(2n - 1 - \frac{1}{2\ell}\right)^2$ and $\left(2n - \frac{1}{2\ell}\right)^2$, respectively.

The proof that $g_{\ell, 0}$ is uniformly bounded is similar, although the first term in \eqref{E:even1} is only bounded by a constant rather than $\sqrt{y}$.  Indeed, in this case the first term is bounded by
$y \sum_{n = 1}^{y^{-1}} n e^{-n^2 y}$, which is bounded by Proposition \ref{P:j=1Sum} to yield the claimed result.

Finally, the case $j=-1$ only arises if $r=3$, in which case $\rho = 0$.  As above, we separate the odd and even terms, and consider only the real part, as the imaginary part can be treated in the same way.  We have
\begin{align}
\label{E:Regl-1}
\text{Re} \left( g_{\ell, -1} (\tau) \right) & =\sum_{n\geq 1} \left[ (2n-1) e^{ - 4 \pi \ell \left( n-\frac12\right)^2 y} \cos \left( 4 \pi\ell  \left( n-\frac12\right)^2 x \right)   - 2n e^{ - 4 \pi \ell n^2 y} \cos \left( 4 \pi\ell  n^2 x \right)\right] \notag \\
& =  2 y^{-\frac12} \sum_{n\geq 1} \left( f_{\frac{x}{y}} \left( \left(n-\frac12 \right) \sqrt{y} \right) - f_{\frac{x}{y}} \left(n\sqrt{y} \right)   \right),
\end{align}
where
$$
f_{v} \left(t \right) := te^{-4\pi\ell t^2} \cos \left( 4\pi \ell v t^2 \right).
$$
Note that $f_v$ is an odd function, and thus its Taylor series only has odd powers of $t$.
Proposition \ref{P:Zag} then gives the asymptotic expansion
\begin{equation}
\label{E:fxy}
 \sum_{n\geq 1} \left( f_{\frac{x}{y}} \left( \left(n-\frac12 \right) \sqrt{y} \right) - f_{\frac{x}{y}} \left(n\sqrt{y} \right)   \right) = \frac{1}{\sqrt{y}} \left( I_f  + O_{\frac{x}{y}} (y)\right) -  \frac{1}{\sqrt{y}} \left( I_f  + O_{\frac{x}{y}} (y)\right)
\end{equation}
so long as
$$
I_f := \int_0^\infty f_v(t) dt = \int_0^\infty te^{-4\pi\ell t^2} \cos \left( 4\pi \ell v t^2 \right) dt
$$
converges.  The integral may be uniformly bounded as
$$
\left| I_f \right| \leq \int_0^\infty te^{-4\pi\ell t^2} dt < \infty,
$$
which, combined with the assumption that $\left| \frac{x}{y} \right|  < 1$, gives an overall bound for \eqref{E:fxy} that is uniform over the claimed region.  The expansions \eqref{E:Regl-1} and \eqref{E:fxy} are therefore uniformly bounded as $O(1)$, which completes the proof of the proposition.
\end{proof}

We conclude by proving an asymptotic expansion used in the proof of Proposition \ref{P:Tq}.
\begin{proposition}
\label{P:gy}
With $g(y)$ defined as in \eqref{E:gy}, as $y \to 0^+$ we have
\begin{equation*}
g(y) = \frac14 y+O\left(y^2\right).
\end{equation*}
\end{proposition}
\begin{proof}
We write
\begin{equation*}
g(y) = y^{\frac12} \sum_{n\geq 1} \left(f \left(\left( n-\frac12 \right) \sqrt{y} \right) - f \left(n\sqrt{y} \right)   \right)
\end{equation*}
with
\begin{align*}
f(t) := \frac{e^{-2t^2} \left(1-e^{-4t^2}  \right)}{2t}=  2t +  O\left(t^3\right).
\end{align*}
Proposition \ref{P:Zag} then gives the expansions
\begin{align*}
&\sum_{n\geq 0} f \left(\left( n+\frac12 \right) \sqrt{y} \right) \sim \frac{I_f}{\sqrt{y}} - B_2 \left( \frac12 \right)  \sqrt{y} + O\left(y^{\frac{3}{2}}\right), \\
&\sum_{n\geq 1} f \left( n \sqrt{y} \right)  \sim \frac{I_f}{\sqrt{y}} - B_2(1)  \sqrt{y} + O\left(y^{\frac{3}{2}}\right),
\end{align*}
with the convergent integral given by
$$
I_f = \int_0^\infty f(u) du < \infty.
$$\
Using the fact that $B_2(x) = x^2 - x + \frac{1}{6},$ we thus find the overall expansion
\begin{equation*}
g(y) = y^{\frac12} \left(\frac1{\sqrt{y}} \left(I_f - I_f \right) + \sqrt{y} \left(\frac1{12} + \frac1{6}\right) + O\left(y^{\frac{3}{2}}\right)\right)= \frac14 y + O\left(y^2\right).
\end{equation*}
\end{proof}

\end{document}